\documentclass[12pt]{amsart}
\usepackage{latexsym}
\usepackage{amssymb}
\usepackage{amsmath}
\usepackage[mathscr]{eucal}
\usepackage[pdftex]{hyperref}
\input xypic
\newtheorem{propo}{Proposition}[section]
\newtheorem{proposition}[propo]{Proposition}

\newtheorem{lemma}[propo]{Lemma}

\newtheorem{theorem}[propo]{Theorem}

\newtheorem{prob}[propo]{Problem}

\newcommand{\Ker}{\operatorname{Ker}}

\newcommand{\Irr}{{\mathrm {Irr}}}
\newcommand{\IBR}{{\mathrm {IBr}}}

\renewcommand{\Im}{{\mathrm {Im}}}

\newcommand{\Hom}{{\mathrm {Hom}}}

\newcommand{\Char}{{\mathrm {char}}}

\newcommand{\CC}{{\mathbb C}}

\newcommand{\FF}{{\mathbb F}}

\begin{document}
\title[permutation modules for $O_{m}^{\pm}(3)$ acting on nonsingular points]
{On the permutation modules for orthogonal\\ groups $O_{m}^{\pm}(3)$
acting on nonsingular points\\ of their standard modules}

\author{Jonathan I. Hall}
\address{Department of Mathematics, Michigan State University, East Lansing,
MI 48824} \email{jhall@math.msu.edu }
\author{Hung Ngoc Nguyen}
\address{Department of Theoretical and Applied Mathematics, University of Akron, Akron,
OH 44253} \email{hn10@uakron.edu} 
\date{\today}

\begin{abstract}
We describe the structure, including composition factors and
submodule lattices, of cross-characteristic permutation modules for
the natural actions of the orthogonal groups $O_{m}^{\pm}(3)$ with
$m\geq6$ on nonsingular points of their standard modules. These
actions together with those studied in~\cite{HN} are all examples of
primitive rank $3$ actions of finite classical groups on nonsingular
points.
\end{abstract}

\maketitle


\section{Introduction}
Given a group $G$ acting on a set $\Omega$ and a field $\FF$, the
problem of determining the structure of the permutation $\FF
G$-module $\FF \Omega$ has been studied extensively for many years.
In particular, permutation modules as well as permutation
representations for finite classical groups have received
significant attention. We are interested in the action of a finite
classical group $G$ on points (i.e. $1$-dimensional subspaces) of
the standard module associated with $G$.

The permutation module for the natural action of $G$ on singular
points has been studied in great depth (for instance,
see~\cite{LST,L1,L2,ST}). However, not much has been known about the
action of $G$ on nonsingular points. We note that in the linear and
symplectic groups, all points are singular.

\begin{prob}\label{low-dim} Let $G$ be a finite orthogonal or unitary group
and $\FF$ an algebraically closed field of cross characteristic. Describe
the submodule structure of the permutation $\FF G$-module for $G$
acting naturally on the set of nonsingular points of its standard
module.
\end{prob}

Let $P^0$ and $P$ be the sets of singular and nonsingular points,
respectively, of the standard module associated with $G$. It is well
known that the action of $G$ on $P^0$ is always transitive and rank
3. On the other hand, the transitivity of the action of $G$ on $P$
and the ranks of that action on orbits depend closely on the
underlying field, which we will denote by~$\FF_q$, of~$G$. We do not
have an exact formula for the ranks but they are ``more or less'' an
increasing function of $q$. In particular, the structure of $\FF P$
becomes more complicated when $q$ is large.

In~\cite{HN}, the authors studied the problem for orthogonal groups
over a field of two elements and unitary groups over a field of four
elements. This is the case (and only case!) when the action of $G$
on $P$ is transitive and rank $3$.

In this article, we study the cross-characteristic permutation
modules $\FF P$ for the orthogonal groups $O_m^\pm(3)$ acting on
$P$. It is not difficult to see that $O_m^\pm(3)$ has two orbits on
$P$ and the action on each orbit is rank 3. These actions together
with those studied in~\cite{HN} are all examples of primitive rank
$3$ actions of finite classical groups on nonsingular points, as
pointed out in an important paper by Kantor and Liebler
(see~\cite{KL}).

Drawing upon the methods introduced in~\cite{HN}, as well as
in~\cite{L1} and~\cite{ST}, we prove the following:

\begin{theorem} \label{theorem}Let $\FF$ be an algebraically closed field
of characteristic $\ell\neq3$. Let $G$ be $O^{\pm}_{m}(3)$ ($m=2n$
or $2n+1$) with $m\geq6$ and $P$ be the set of nonsingular points of
the standard module associated with $G$. Then the permutation $\FF
G$-module $\FF P$ of $G$ acting naturally on $P$ has the submodule
structure as described in Tables~1,2, and~3. In these tables,
$\delta_{i,j}=1$ if $i\mid j$ and $0$ otherwise.
\end{theorem}

\footnotesize
\begin{table}[h]\label{tableO+}\caption{Submodule structure of $\FF O^+_{2n}(3)$-module $\FF P$.}
\begin{tabular}{ll} \hline
Conditions on $\ell$ and $n$\quad\quad\quad & Structure of $\FF
P$\\\hline
\xymatrix@C=8pt@R=8pt{\ell\neq2,3;\ell\nmid (3^n-1)}& \xymatrix@C=8pt@R=8pt{2\FF\oplus X\oplus Y\oplus 2Z }\\

\xymatrix@C=8pt@R=8pt{\ell\neq2,3;\ell\mid
(3^n-1)}&\xymatrix@C=8pt@R=8pt{
&&&&\FF\ar@{-}[d]&&\FF\ar@{-}[d]\\
X&\oplus&Y&\oplus&Z\ar@{-}[d]&\oplus&Z\ar@{-}[d]\\
&&&&\FF&&\FF}\\

\xymatrix@C=8pt@R=8pt{\ell=2;n \text{ even }}&\xymatrix@C=8pt@R=8pt{
\FF\ar@{-}[d]&&X\ar@{-}[dll]\ar@{-}[d]&&\FF\ar@{-}[d]&&Y\ar@{-}[dll]\ar@{-}[d]\\
W\ar@{-}[d]\ar@{-}[drr]&&Y\ar@{-}[d]&\oplus&W\ar@{-}[d]\ar@{-}[drr]&&X\ar@{-}[d]\\
\FF&&X&&\FF&&Y}\\

\xymatrix@C=8pt@R=8pt{\ell=2;n \text{ odd}}& \xymatrix@C=8pt@R=8pt{
&&&X\ar@{-}[dl]\ar@{-}[dr]&&&&Y\ar@{-}[dl]\ar@{-}[dr]&\\
2\FF&\oplus&Y\ar@{-}[dr]&&W\ar@{-}[dl]&\oplus&X\ar@{-}[dr]&&W\ar@{-}[dl]\\
&&&X&&&&Y&} \\\hline
\end{tabular}

\footnotesize where, $\dim X=\dim Y=\frac{(3^n-1)(3^{n-1}-1)}{8}$,
$\dim Z=\frac{3^{2n}-9}{8}-\delta_{\ell,3^n-1}$,\\ and $\dim
W=\frac{(3^n-1)(3^{n-1}+3)}{8}-1-\delta_{2,n}$.\normalsize
\end{table}

\begin{table}[h]\caption{Submodule structure of $\FF O^-_{2n}(3)$-module $\FF P$.}\label{tableO-}
\begin{tabular}{ll} \hline
Conditions on $\ell$ and $n$\quad\quad\quad & Structure of $\FF
P$\\\hline
\xymatrix@C=8pt@R=8pt{\ell\neq2,3;\ell\nmid (3^n+1)}& \xymatrix@C=8pt@R=8pt{2\FF\oplus X\oplus Y \oplus 2Z}\\

\xymatrix@C=8pt@R=8pt{\ell\neq2,3;\ell\mid
(3^n+1)}&\xymatrix@C=8pt@R=8pt{
&&&&\FF\ar@{-}[d]&&\FF\ar@{-}[d]\\
X&\oplus&Y&\oplus&Z\ar@{-}[d]&\oplus&Z\ar@{-}[d]\\
&&&&\FF&&\FF}\\

\xymatrix@C=8pt@R=8pt{\ell=2;n \text{ even}}
&\xymatrix@C=3pt@R=3pt{&&&X\ar@{-}[ddl]\ar@{-}[dr]&&&&Y\ar@{-}[ddl]\ar@{-}[dr]&\\
&&&&\FF\ar@{-}[d]&&&&\FF\ar@{-}[d]\\
2\FF&\oplus&W\ar@{-}[ddr]&& Y\ar@{-}[d]&\oplus& W\ar@{-}[ddr]&&X\ar@{-}[d]\\
&&&&\FF\ar@{-}[dl]&&&&\FF\ar@{-}[dl]\\
&&&X&&&&Y&}\\

\xymatrix@C=8pt@R=8pt{\ell=2;n \text{ odd
}}&\xymatrix@C=3pt@R=3pt{\FF\ar@{-}[dd]&&X\ar@{-}[ddll]\ar@{-}[d]&&\FF\ar@{-}[dd]&&Y\ar@{-}[ddll]\ar@{-}[d]\\
&&\FF\ar@{-}[d]&&&&\FF\ar@{-}[d]\\
Y\ar@{-}[dd]\ar@{-}[ddrr]&&W\ar@{-}[d]&\oplus&X\ar@{-}[dd]\ar@{-}[ddrr]&&W\ar@{-}[d]\\
&&\FF\ar@{-}[d]&&&&\FF\ar@{-}[d]\\
\FF&&X&&\FF&&Y }\\\hline
\end{tabular}

\footnotesize where, $\dim X=\dim
Y=\frac{(3^n+1)(3^{n-1}+1)}{8}-\delta_{\ell,2}$, $\dim
Z=\frac{3^{2n}-9}{8}-\delta_{\ell,3^n+1}$, \\and $\dim
W=\frac{(3^n+1)(3^{n-1}-3)}{8}-1+\delta_{2,n}$.\normalsize
\end{table}

\begin{table}[h]\caption{Submodule structure of $\FF O_{2n+1}(3)$-module $\FF P$.}\label{tableOdd}
\begin{tabular}{ll} \hline
Conditions on $\ell$ and $n$\quad\quad\quad\quad\quad & Structure of
$\FF P$\\\hline
\xymatrix@C=8pt@R=8pt{\ell\neq2,3;\ell\nmid (3^n-1),\ell\nmid (3^n+1)}& \xymatrix@C=8pt@R=8pt{2\FF\oplus X\oplus Y\oplus 2Z }\\

\xymatrix@C=8pt@R=8pt{\ell\neq2,3;\ell\mid
(3^n-1)}&\xymatrix@C=8pt@R=8pt{
&&\FF\ar@{-}[d]&&&&\\
\FF&\oplus &X\ar@{-}[d]&\oplus&Y&\oplus&2Z\\
&&\FF&&&&}\\

\xymatrix@C=8pt@R=8pt{\ell\neq2,3;\ell\mid
(3^n+1)}&\xymatrix@C=8pt@R=8pt{
&&&&\FF\ar@{-}[d]&&\\
\FF&\oplus &X&\oplus&Y\ar@{-}[d]&\oplus&2Z\\
&&&&\FF&&}\\

\xymatrix@C=8pt@R=8pt{\ell=2;n \text{ even
}}&\xymatrix@C=3pt@R=3pt{&&\FF\ar@{-}[dd]&&X_1\ar@{-}[dd]\ar@{-}[ddll]&&&Y_1\ar@{-}[ddl]\ar@{-}[dr]&\\
&&&&&&&&\FF\ar@{-}[d]\\
\FF&\oplus&Z_1\ar@{-}[dd]\ar@{-}[ddrr]&&Y_1\ar@{-}[dd]&\oplus&X_1\ar@{-}[ddr]&&Z_1\ar@{-}[d]\\
&&&&&&&&\FF\ar@{-}[dl]\\
&&\FF&&X_1&&&Y_1&}\\

\xymatrix@C=8pt@R=8pt{\ell=2;n \text{
odd}}&\xymatrix@C=8pt@R=8pt{&&&X_1\ar@{-}[dl]\ar@{-}[dr]&&&\FF\ar@{-}[ddr]&Y_1\ar@{-}[d]\ar@{-}[dr]\ar@{-}[ddl]&\\
\FF&\oplus&Z_1\ar@{-}[dr]&&Y_1\ar@{-}[dl]&\oplus&&Z_1\ar@{-}[d]&X_1\ar@{-}[dl]\\
&&&X_1&&&\FF&Y_1& }\\\hline
\end{tabular}

\footnotesize where, $\dim
X=\frac{(3^n+1)(3^{n}-3)}{4}-\delta_{\ell,3^n-1}$, $\dim
Y=\frac{(3^n-1)(3^{n}+3)}{4}-\delta_{\ell,3^n+1}$, $\dim
Z=\frac{3^{2n}-1}{4}$, $\dim X_1=\frac{(3^n-1)(3^n-3)}{8}$, $\dim
Y_1=\frac{(3^n+1)(3^n+3)}{8}-1$, and $\dim
Z_1=\frac{(3^{2n}-9)}{8}-\delta_{2,n}$.\normalsize
\end{table}

\normalsize

Let $V$ be a vector space of dimension $m\geq 6$ over the field of
$3$ elements $\mathbb{F}_3=\{0,1,-1\}$. Let $Q$ be a non-degenerate
quadratic form on $V$, and let $(\cdot,\cdot)$ be the non-degenerate
symmetric bilinear form on $V$ associated with $Q$ so that
$Q(au+bv)=a^2Q(u)+b^2Q(v)+ab(u,v)$ for any $a,b\in \FF_3, u,v\in V$.
Then $G=O_m^\pm(3)$ is the full orthogonal group consisting of all
linear transformations of $V$ preserving $Q$.

For $\kappa=\pm1$, we denote
$$P^\kappa:=P^\kappa(Q):=\{\langle v\rangle\in P\mid Q(v)=\kappa\}.$$ $P^{+1}$ and $P^{-1}$ are often called the sets
of plus points and minus points, respectively. We obtain the
following isomorphism of $\FF G$-modules:
$$\FF P\cong\FF P^{+1}\oplus\FF P^{-1}.$$

Since $Q$ is a non-degenerate quadratic form on $V$, $-Q$ is also a
non-degenerate quadratic form on $V$. The two isometry groups $O(V,
Q)$ and $O(V,-Q)$ are canonically isomorphic, but the corresponding
sets of nonsingular points are switched: $P^{-1}(Q)=P^{+1}(-Q)$ and
$P^{+1}(Q)=P^{-1}(-Q)$. When $m$ is even the forms $Q$ and $-Q$ have
the same discriminant but when $m$ is odd they do not. Therefore in
Theorem~\ref{theorem} and its associated tables we see that for $m$
even there are two distinct isometry groups $G$, but for each the
modules $\FF P^{+1}$ and $\FF P^{-1}$ are the same up to a diagonal
automorphism of $G$, whereas for $m$ odd there is only one isometry
group to consider, but the modules $\FF P^{+1}$ and $\FF P^{-1}$ are
fundamentally different and indeed have different dimensions.

The paper is organized as follows. In the next section, we will
outline the proof of the main theorem. $\S$\ref{minimality} and
$\S$\ref{relations} are some preparations for the following
sections. Each family of groups $O_{2n}^+(3)$, $O_{2n}^-(3)$, and
$O_{2n+1}(3)$ is treated separately in sections~$\S$\ref{sectionO+},
$\S$\ref{sectionO-}, and~$\S$\ref{sectionOdd}, respectively.
\section{Notation and outline of the proof}

\subsection{Preliminaries}

If the action of $G$ on a set $\Omega$ is rank $3$ then the $\FF
G$-module $\FF\Omega$ has two special submodules, the so-called
\emph{graph submodules}. The following description of these graph
submodules is due to Liebeck (see~\cite{L1}).

Let $G_\alpha$ be the stabilizer of $\alpha\in \Omega$. Then
$G_\alpha$ acts on $\Omega$ with $3$ orbits: one of them is
$\{\alpha\}$ and the other two are denoted by $\Delta(\alpha)$ and
$\Phi(\alpha)$. Define the parameters: $a=|\Delta(\alpha)|$,
$b=|\Phi(\alpha)|$, $r=|\Delta(\alpha)\cap\Delta(\beta)|$, and
$s=|\Delta(\alpha)\cap\Delta(\gamma)|$ for $\beta\in \Delta(\alpha)$
and $\gamma\in \Phi(\alpha)$. For any subset $\Delta$ of $\Omega$,
we denote by $[\Delta]$ the element $\Sigma_{\delta\in\Delta}
\delta$ of $\FF \Omega$. For $c\in \FF$, let $U_c$ be the $\FF
G$-submodule of $\FF \Omega$ generated by all elements
$v_{c,\alpha}=c\alpha+[\Delta(\alpha)], \alpha\in \Omega$ and $U'_c$
be the $\FF G$-submodule of $U_c$ generated by all elements
$v_{c,\alpha}-v_{c,\beta}=c(\alpha-\beta)+[\Delta(\alpha)]-[\Delta(\beta)],
\alpha, \beta\in\Omega$. The \emph{graph submodules} of the
permutation $\FF G$-module $\FF \Omega$ are defined to be $U'_{c_1}$
and $U'_{c_2}$ where $c_1$ and $c_2$ are the roots of the quadratic
equation:
\begin{equation}\label{quadratic} x^2+(r-s)x+s-a=0.
\end{equation}

Setting $S(\FF \Omega)=\{\sum_{\omega\in\Omega}a_\omega\omega\mid
a_\omega\in\FF, \sum a_\omega=0\}$ and $T(\FF
\Omega)=\{c[\Omega]\mid c\in\FF\}$. $S(\FF\Omega)$ and
$T(\FF\Omega)$ are $\FF G$-submodules of $\FF \Omega$ of dimensions
$|\Omega|-1,1$, respectively. Moreover, $T(\FF \Omega)$ is
isomorphic to the one-dimensional trivial module.

Suppose that two graph submodules are different, i.e., $c_1\neq
c_2$. Since $v_{c_1,\alpha}-v_{c_2,\alpha}=(c_1-c_2)\alpha$ for any
$\alpha\in \Omega$, we have \begin{equation}\label{direct
sum}U'_{c_1}\oplus U'_{c_2}=S(\FF \Omega).\end{equation} The modules
$U'_{c_1}$ and $U'_{c_2}$ are the eigenspaces of the linear
transformation $T:\FF \Omega \rightarrow \FF \Omega$ defined by
$T(\alpha)=[\Delta(\alpha)], \alpha\in \Omega$ corresponding to
eigenvalues $-c_2$ and $-c_1$, respectively (see~\cite{HN}). As $T$
has trace $0$ with $T(\FF \Omega)$ an eigenspace for the eigenvalue
$a$, we can compute the dimensions of graph submodules from the
following equations:
\begin{equation}\label{dimension} \left\{\begin {array}{l}
\dim U'_{c_1}+\dim U'_{c_2}=|\Omega|-1,\\
c_2\dim U'_{c_1}+c_1\dim U'_{c_2}=a.
 \end {array} \right.
\end{equation}

We note that $\FF\Omega$ has a nonsingular and $G$-invariant inner
product defined by $\big\langle
\sum_{\omega\in\Omega}a_{\omega}\omega,\sum_{\omega\in\Omega}b_{\omega}
\omega\big\rangle=\sum_{\omega\in\Omega}a_{\omega}b_\omega.$ If $U$
is a submodule of $\FF \Omega$, we denote by $U^\perp$ the submodule
of $\FF \Omega$ consisting of all elements orthogonal to $U$. We
need the following result, which is due to Liebeck and is stated as
Lemma~2.1 in~\cite{HN}.

\begin{lemma}[\cite{L1}] \label{liebeck} If $c$ is not a root of
equation (\ref{quadratic}) then $U'_c=S(\FF\Omega)$. Moreover, if
$c_1$ and $c_2$ are roots of this equation then $\langle
v_{c_1,\alpha},v_{c_2,\beta}\rangle =s$ for any $\alpha, \beta\in
\Omega$. Consequently, $\langle U'_{c_1}, U_{c_2}\rangle=\langle
U'_{c_2}, U_{c_1}\rangle=0$.
\end{lemma}

\subsection{Outline of the proof} We first compute the roots of equation (\ref{quadratic}) and then
determine the graph submodules of $\FF P^{+1}$ and $\FF P^{-1}$ by
analyzing the geometry of $P$. As in the study of rank 3 permutation
modules in cross-characteristic for finite classical group acting on
singular points in~\cite{L1,L2}, the graph submodules are ``minimal"
in an appropriate sense (see Proposition~\ref{mainO+}).

The problem now is divided in two cases. In the easy case when the
graph submodules are different, we have seen from~(\ref{direct sum})
that their direct sum is $S(\FF P^\kappa)$ $(\kappa=\pm1)$, a
submodule of $\FF P^\kappa$ of codimension~$1$. Therefore the full
structure of $\FF P=\FF P^{+1}\oplus \FF P^{-1}$ can be determined
without significant effort.

The difficult case is when the graph submodules are the same (i.e.
$c_1=c_2$). We will see later that the graph submodules of $\FF
P^{+1}$ are equal if and only if those of $\FF P^{-1}$ as well as
$\FF P^0$ are equal and this happens when $\ell=\Char(\FF)=2$. We
handle this case by constructing some relations between $\FF P^0$,
$\FF P^{+1}$, and $\FF P^{-1}$ in~$\S$\ref{relations}.

\subsection{Further notation}

Given a finite group $G$, $\Irr(G)$ and $\IBR_{\ell}(G)$) will be
the sets of irreducible complex characters and irreducible
$\ell$-Brauer characters, respectively, of~$G$. For
$\chi\in\Irr(G)$, by $\overline{\chi}$ we always mean its reduction
modulo $2$. Following~\cite{ST}, we denote by $\beta(M)$ the Brauer
character of $G$ afforded by an $\FF G$-module $M$. Furthermore, if
$\beta(M)\in \IBR_{\ell}(G)$ is a constituent of an $\ell$-Brauer
character $\varphi$, we say that $M$ is a constituent of $\varphi$.
Finally, if $H$ is a subgroup of $G$, we denote by $M|_{H}$ the
restriction of $M$ to $H$.


\section{Minimality of the graph submodules of $\FF P^{+1}$ and $\FF P^{-1}$}\label{minimality}

First we fix a basis $\mathcal{B}$ of $V$. If $Q$ is a quadratic
form of type $+$ on $V$ with dimension $2n$, we consider
$\mathcal{B}=\{e_1,...,e_n,f_1,...,f_n\}$ so that
$(e_i,f_j)=\delta_{ij}$ and $(e_i,e_j)=(f_i,f_j)=0$ for
$i,j=1,...,n$. If $Q$ is of type $-$ on $V$ with dimension $2n$,
then $\mathcal{B}=\{e_1,...,e_n,f_1,...,f_n\}$ where
$(e_i,f_j)=\delta_{ij}$ and $(e_i,e_n)=(f_i,f_n)=0$ for
$i,j=1,...,n-1$, $(e_n, f_n)=0$, and $(e_n,e_n)=(f_n,f_n)=1$.
Finally, if $\dim V=2n+1$, then
$\mathcal{B}=\{e_1,...,e_n,f_1,...,f_n,g\}$ where
$(e_i,f_j)=\delta_{ij}$, $(e_i,e_j)=(f_i,f_j)=(e_i,g)=(f_i,g)=0$,
and $(g,g)=1$ for $i,j=1,...,n$.

Going back to the action of $G$ on $P^\kappa, \kappa=\pm1$, we
assume from now on that $\Delta(\alpha)\subset P^\kappa\setminus
\alpha$ consists of points orthogonal to $\alpha$ and
$\Phi(\alpha)\subset P^\kappa\setminus \alpha$ consists of points
not orthogonal to $\alpha$. Also, we use the notation $U^\kappa_c$
and $U^{'\kappa}_c$ for $U_c$ and $U'_c$, respectively. The graph
submodules of $\FF P^\kappa$ now are $U^{'\kappa}_{c_1}$ and
$U^{'\kappa}_{c_2}$ where $c_1$ and $c_2$ are roots of the
equation~(\ref{quadratic}). At this point we understand that $c_1$
and $c_2$ depend on~$\kappa$ but actually they do not, as we will
see later on.

As in the study of rank 3 permutation modules in
cross-characteristic for finite classical group acting on singular
points in~\cite{L1,L2}, the graph submodules are ``minimal" in the
following sense:

\begin{proposition}\label{mainO+} Suppose that $\Char(\FF)\neq3$.
Then every nonzero $\FF G$-submodule of $\FF P^\kappa (\kappa=\pm1)$
either is $T(\FF P^\kappa)$ or contains a graph submodule.
\end{proposition}

\begin{proof} We only give here the proof for the case
$G=O^+_{2n}(3)$ and $\kappa=+1$. Other cases are similar. We partly
follow some ideas and notation from~\cite{HN,L1,L2}.

Let $\phi_1:=\langle e_2+f_2\rangle$, $\phi_2:=\langle
e_1+e_2+f_2\rangle$, and $\phi_3:=\langle -e_1+ e_2+f_2\rangle$. Let
$\Delta_1:=\{\langle\sum_{i=1}^n(a_ie_i+b_if_i)\rangle\in P^{+1}\mid
b_1=1, a_2+b_2=0\}$,
$\Delta_2:=\{\langle\sum_{i=1}^n(a_ie_i+b_if_i)\rangle\in P^{+1}\mid
b_1=1, a_2+b_2=-1\}$,
$\Delta_3:=\{\langle\sum_{i=1}^n(a_ie_i+b_if_i)\rangle\in P^{+1}\mid
b_1=a_2+b_2=1\}$, $\Delta:=\Delta_1\cup \Delta_2\cup\Delta_3$, and
$\Phi:=P^{+1}\setminus \Delta$. It is clear that, for $i,j=1,2,3$,
\begin{equation}\label{Deltaphi}[\Delta(\phi_i)]-[\Delta(\phi_j)]=[\Delta_i]-[\Delta_j].\end{equation}
Consider a subgroup $H<G$ consisting of orthogonal transformations
sending elements of the basis $\{e_1,f_1,e_2,f_2,...,e_n,f_n\}$ to
those of basis $\{e_1,f_1+\sum_{i=1}^na_ie_i+\sum_{i=2}^nb_if_i$,
$e_2-b_2e_1,f_2-a_2e_1,...,e_n-b_ne_1,f_n-a_ne_1\}$ respectively,
where $a_i,b_i\in\FF_3$ and $-a_1=\sum_{i=2}^na_ib_i$. In other
words, $H$ is subgroup of isometries fixing $\langle e_1\rangle$ and
acting trivially on each factor of the series $0\leq \langle
e_1\rangle\leq \langle e_1\rangle^\perp\leq V$. Let $K$ be the
subgroup of $H$ consisting of transformations fixing $\phi_1$. Let
$P^{+1}_1$ be the set of plus points in $V_1=\langle
e_2,f_2,...,e_n,f_n\rangle$. For each $\langle w\rangle\in
P^{+1}_1$, define $B_{\langle w\rangle}=\{\langle w\rangle, \langle
e_1+w\rangle,\langle -e_1+w\rangle\}$. As in Propositions~2.1 and
2.2 of~\cite{L1} and Lemmas~3.2 and 3.3 of \cite{HN}, we have
\begin{enumerate}
\item[(i)] $|H|=3^{2n-2}$, $|K|=3^{2n-3}$, $|\Delta|=3^{2n-2}$, and $|\Delta_1|=|\Delta_2|=|\Delta_3|=3^{2n-3}$;
\item[(ii)] $H$ acts regularly on $\Delta$ and $K$ has $3$ orbits
$\Delta_1,\Delta_2, \Delta_3$ on $\Delta$;
\item[(iii)] $\Phi=\bigcup_{\langle w\rangle\in P^{+1}_1}B_{\langle w\rangle}$;
\item[(iv)] $K$ fixes $B_{\phi_1}$ point-wise and is transitive
on $B_w$ for every $\phi_1\neq\langle w\rangle\in P^{+1}_1$;
\item[(v)] $H$ acts transitively on $B_{\langle w\rangle}$ for every $\langle w\rangle\in P^{+1}_1$.
\end{enumerate}

Suppose that $U$ is a nonzero submodule of $\FF P^{+1}$. Assume
$U\neq T(\FF P^{+1})$, so that $U$ contains an element of the form
$$u=a\langle x\rangle+b\langle y\rangle+\sum_{\delta\in P^{+1}\backslash\{\langle x\rangle,\langle y\rangle\}}a_\delta\delta,$$
where $a,b,a_\delta\in\FF$ and $a\neq b$. If $(x,y)=0$, we choose an
element $\langle z\rangle\in P^{+1}$ so that $(x,z)$ and $(y,z)$ are
nonzero. Since $a\neq b$, the coefficient of $\langle z\rangle$ in
$u$ is different from either $a$ or $b$. Therefore, with no loss, we
may assume $(x,y)\neq 0$. Since $(e_2+f_2,e_1+e_2+f_2)\neq0$, there
exists $g'\in G$ such that $\langle x\rangle g'=\phi_1$ and $\langle
y\rangle g'=\phi_2$. Therefore, we can assume that
$u=a\phi_1+b\phi_2+\sum_{\delta\in
P^{+1}\backslash\{\phi_1,\phi_2\}}a_\delta\delta$.

Let $g\in G$ such that $e_1g=e_1$ and $(e_2+f_2)g=-(e_1+e_2+f_2)$.
Then $\phi_1 g=\phi_2$, $\phi_2 g=\phi_1$, $\phi_3 g=\phi_3$, and
therefore
$$u-ug=(a-b)(\phi_1-\phi_2)+\sum_{\delta\in P^{+1}\backslash\{\phi_1,\phi_2,\phi_3\}}b_\delta\delta\in
U\cap S(\FF P^{+1}),$$ where $b_\delta\in\FF$. Note that $u-ug\in
S(\FF P^{+1})$. Therefore, if $c_\delta=b_\delta/(a-b)$, we get $$
u_1:=(u-ug)/(a-b)=\phi_1-\phi_2+\sum_{\delta\in
P^{+1}\backslash\{\phi_1,\phi_2,\phi_3\}}c_\delta\delta\in U\cap
S(\FF P^{+1}).
$$
Hence we have $u_2:=\sum_{k\in K}u_1k\in U\cap S(\FF P^{+1})$.
Moreover,
$$u_2=3^{2n-3}(\phi_1-\phi_2)+\sum_{\delta\in\Delta}d_\delta\delta+\sum_{\langle
w\rangle\in P^{+1}_1, \langle w\rangle\neq \phi_1}d_{\langle
w\rangle}[B_{\langle w\rangle}],$$ where $d_\delta,d_{\langle
w\rangle}\in\FF$. Therefore $u_3:=\sum_{h\in H}u_2h\in U\cap S(\FF
P^{+1})$ with
$$u_3=(\sum_{\delta\in\Delta}d_\delta)[\Delta]+3^{2n-2}\sum_{\langle w\rangle\in P^{+1}_1,
\langle w\rangle\neq \phi_1}d_{\langle w\rangle}[B_{\langle
w\rangle}].$$ It follows that
$$u_4:=3^{2n-2}u_2-u_3=3^{4n-5}(\phi_1-\phi_2)+\sum_{\delta\in\Delta}f_\delta\delta\in
U\cap S(\FF P^{+1}),$$ where
$f_\delta=3^{2n-2}d_\delta-\sum_{\delta\in\Delta}d_\delta$. Hence
$$u_5:=\sum_{k\in
K}u_4k=3^{6n-8}(\phi_1-\phi_2)+f[\Delta_1]+f'[\Delta_2]+f''[\Delta_3]\in
U\cap S(\FF P^{+1}),$$ where $f,f',f''\in \FF$. In particular,
$f+f'+f''=0.$

{\bf Case 1:} $f+f'=-f''=0$. Then
$u_5=3^{6n-8}(\phi_1-\phi_2)+f[\Delta_1]+f'[\Delta_2]=3^{6n-8}(\phi_1-\phi_2)+f([\Delta_1]-[\Delta_2])=
3^{6n-8}(\phi_1-\phi_2)+f([\Delta(\phi_1)]-[\Delta(\phi_2)])\in U$
by~(\ref{Deltaphi}). Assume that $f=0$. Then
$u_5=3^{6n-8}(\phi_1-\phi_2)\in U$. It follows that
$\phi_1-\phi_2\in U$. Hence $\alpha-\beta\in U$ for every
$\alpha,\beta\in P^{+1}$ and therefore $U\supseteq S(\FF P^{+1})$,
which implies that $U$ contains a graph submodule.

It remains to consider $f\neq 0$. Then we have
$(3^{6n-8}/f)(\phi_1-\phi_2)+[\Delta(\phi_1)]-[\Delta(\phi_2)]\in
U$. It follows that
$(3^{6n-8}/f)(\alpha-\beta)+[\Delta(\alpha)]-[\Delta(\beta)]\in U$
for every $\alpha,\beta\in P^{+1}$ and hence $U\supseteq
U'_{3^{6n-8}/f}$, which implies that $U$ contains a graph submodule
by Lemma~\ref{liebeck}.

{\bf Case 2:} $f+f'\neq 0$. Define an element $g\in G$ which sends
elements of the basis $\{e_1$, $f_1$, $e_2$, $f_2,...,e_n,f_n\}$ to
those of basis $\{e_1,f_1-f_2,-e_2-e_1,-f_2,e_3,f_3,...,e_n,f_n\}$
respectively. It is easy to check that $\phi_1g=\phi_2$,
$\phi_2g=\phi_1$, and $\phi_3g=\phi_3$. Also, $\Delta_1g=\Delta_2$,
$\Delta_2g=\Delta_1$, and $\Delta_3g=\Delta_3$. So we have
$$u_6:=u_5-u_5g=2\cdot3^{6n-8}(\phi_1-\phi_2)+(f+f')([\Delta_1]-[\Delta_2])\in
U.$$ As above, we obtain $U\supseteq U'_{2\cdot3^{6n-8}/(f+f')}$,
which again implies that $U$ contains a graph submodule, as desired.
\end{proof}

We will see later for $\ell=2$ that $c_1=c_2$ and hence $\FF
P^\kappa$ has a unique graph submodule. In this case we set
$U^\kappa:=U^\kappa_{c_1}$ and $U^{'\kappa}:=U^{'\kappa}_{c_1}$. The
following lemma is an important property of the graph submodule and
is useful in determining the modulo $2$ structure of the permutation
module.

\begin{lemma}\label{lemmaO+1} For $\ell=2$ and $\kappa=\pm1$,
\begin{enumerate}
\item[(i)] $U^\kappa=T(\FF P^\kappa)\oplus U^{'\kappa}$. In particular, by
Proposition~\ref{mainO+}, $U^{'\kappa}$ is simple and $U^\kappa$ is
the socle of $\FF P^\kappa$;
\item[(ii)] $U^\kappa$ and
$U^{'\kappa}$ are self-dual. Furthermore, $U^{'\kappa}$ appears at
least twice as a composition factor of $\FF P^\kappa$.
\end{enumerate}
\end{lemma}

\begin{proof} (i) If $G=O^-_{2n}(3)$ and $\kappa=-1$, let $S\subset P^\kappa$ be the set
of points of the form $\langle e_n+v\rangle$ where $v\in\langle
e_1,...,e_{n-1}\rangle$. In all other cases, let $S\subset P^\kappa$
be the set of points of the form $\langle e_n+\kappa f_n+v\rangle$
where $v\in\langle e_1,...,e_{n-1}\rangle$. We then have
$$\sum_{\alpha\in S}v_{1,\alpha}=\sum_{\alpha\in S}(\alpha+[\Delta(\alpha)])=
\sum_{\alpha\in S}\alpha+3^{i(\beta)}\sum_{\beta\in
P^\kappa\setminus S}\beta=[P^\kappa],$$ where $i(\beta)=n-2$ or
$n-1$. Hence $[P^\kappa]\in U^\kappa$ or equivalently $T(\FF
P^\kappa)\subset U^\kappa$. As $|S|=3^{n-1}\neq 0$ in $\FF$, $T(\FF
P^\kappa)\nsubseteq U^{'\kappa}$. Since $U^{'\kappa}$ is a submodule
of $U^{\kappa}$ of codimension at most $1$, $U^\kappa=T(\FF
P^\kappa)\oplus U^{'\kappa}$. This and Proposition~\ref{mainO+} show
that $U^{'\kappa}$ is simple and $U^\kappa$ is the socle of $\FF
P^\kappa$.

(ii) Recall that the submodule $U^\kappa$ consists of $\FF$-linear
combinations of $v_{c,\alpha}$ where $\alpha\in P^{\kappa}$ and
$c=c_1=c_2$. Define a bilinear form $[\cdot,\cdot]$ on $U^{\kappa}$
by $[v_{c,\alpha},v_{c,\beta}]=\langle v_{c,\alpha},\beta\rangle$.
It is clear that this form is symmetric, non-singular, and
$G$-invariant. Hence, $U^{\kappa}$ is self-dual. It then follows
that $U^{'\kappa}$ is also self-dual by~(i). Therefore $\FF
P^\kappa/U^{'\kappa\perp}\cong \Hom_\FF(U^{'\kappa\perp}, \FF)\cong
U^{'\kappa\perp}$. Combining this with the inclusion
$U^{'\kappa}\subseteq U^{'\kappa\perp}$ (by Lemma~\ref{liebeck}), we
have that $U^{'\kappa}$ appears at least twice as a composition
factor of $\FF P^\kappa$.
\end{proof}


\section{Relations between $\FF P^0$, $\FF P^{+1}$ and $\FF
P^{-1}$}\label{relations}

In this section, we establish some relations between the structures
of $\FF P^0$, $\FF P^{+1}$ and $\FF P^{-1}$. This helps us to
understand $\FF P^{+1}$ and $\FF P^{-1}$ from the known results of
$\FF P^0$ in~\cite{L2,ST}.

For $i,j\in \{0,+1,-1\}$, define
\begin{equation}\label{Q}\begin{array}{cccc}Q_{i,j}:&\FF P^i &\rightarrow &\FF P^j\\
&\alpha&\mapsto &[\{\beta\in
P^j\mid\beta\perp\alpha\}].\end{array}\end{equation} It is obvious
that $Q_{i,j}$ is an $\FF G$-homomorphism. Also, $Q_{i,j}(S(\FF
P^i))\subseteq S(\FF P^j)$. The following lemma is easy to check.

\begin{lemma}\label{main1} $\Im(Q_{i,j})$ and $\Im (Q_{i,j}|_{S(\FF P^i)})$ are
nonzero and different from $T(\FF P^j)$. Also, $\FF
P^i/\Ker(Q_{i,j})\cong \Im(Q_{i,j})$ and $S(\FF
P^i)/\Ker(Q_{i,j}|_{S(\FF P^i)})\cong \Im(Q_{i,j}|_{S(\FF
P^i)})$.\hfill$\Box$
\end{lemma}

Let $\rho^0$, $\rho^{+1}$, and $\rho^{-1}$ be the complex
permutation characters of $G$ afforded by permutation modules $\CC
P^0$, $\CC P^{+1}$, and $\CC P^{-1}$, respectively. Recall that $G$
acts with rank $3$ on each $P^i$ with $i=0,+1,-1$. Hence $\rho^i$
has $3$ constituents, all of multiplicity $1$ and exactly one of
them is the trivial character.

\begin{lemma}\label{main2} For $i,j\in \{0,+1,-1\}$, $\rho^i$ and $\rho^j$ have a
common nontrivial constituent.
\end{lemma}

\begin{proof}This is an immediate consequence of Lemma~\ref{main1}.
\end{proof}


\section{The orthogonal groups $O^+_{2n}(3)$}\label{sectionO+}

In this section, we always assume $G=O^{+}_{2n}(3)$. For
$\kappa=\pm1$, we have
$P^\kappa=\{\langle\sum_{i=1}^n(a_ie_i+b_if_i)\rangle\mid a_i,
b_i\in \FF_3, \sum_{i=1}^n a_ib_i=\kappa\}$ and
$|P^\kappa|=3^{n-1}(3^n-1)/2$. The parameters of the action of $G$
on $P^\kappa$ are:
$$a=\frac{3^{n-1}(3^{n-1}-1)}{2}, b=3^{2n-2}-1, r=\frac{3^{n-2}(3^{n-1}+1)}{2}, s=\frac{3^{n-1}(3^{n-2}-1)}{2}.$$
The equation~(\ref{quadratic}) now has two roots $3^{n-2}$ and
$-3^{n-1}$. Therefore, $\FF P^\kappa$ has graph submodules
$U^{'\kappa}_{3^{n-2}}$ and $U^{'\kappa}_{-3^{n-1}}$.

We note that if $\ell=\Char(\FF)\neq 2,3$ then $3^{n-2}\neq
-3^{n-1}$ and therefore two graph submodules are different.

\begin{lemma}\label{dimensionO+} If $\ell=\Char(\FF)\neq 2,3$,
then
$$\dim U^{'\kappa}_{3^{n-2}}=\frac{(3^n-1)(3^{n-1}-1)}{8},
\dim U^{'\kappa}_{-3^{n-1}}=\frac{3^{2n}-9}{8}\text{ and }
U^{'+1}_{-3^{n-1}}\cong U^{'-1}_{-3^{n-1}}.$$
\end{lemma}

\begin{proof} The dimensions of $U^{'\kappa}_{3^{n-2}}$ and $U^{'\kappa}_{-3^{n-1}}$
follow from~(\ref{dimension}). Using results about the permutation
module for $G$ acting on $P^0$ in~\cite{L2}, we see that $\FF P^0$
has two graph submodules of dimensions $(3^n-1)(3^{n-1}+3)/8$ and
$(3^{2n}-9)/8$, which we temporarily denote by $U'_c$ and $U'_d$,
respectively. By Theorem~2.1 of~\cite{L2}, $U'_c$ and $U'_d$ are
minimal in the same sense as in Proposition \ref{mainO+}. Applying
Lemma~\ref{main1} and Proposition~\ref{mainO+}, we deduce that
$U^{'+1}_{-3^{n-1}}\cong U'_d$ and $U^{'-1}_{-3^{n-1}}\cong U'_d$
and the lemma follows.
\end{proof}

\begin{proposition}\label{proofO+1}
Theorem~\ref{theorem} holds when $G=O^+_{2n}(3),n\geq3$ and
$\ell\neq 2,3$.
\end{proposition}

\begin{proof}
First, consider $\ell\nmid (3^n-1)$. Then $[P^\kappa]\notin S(\FF
P^\kappa)$ and we have
$$\begin{array}{lll}\FF P&\cong &\FF P^{+1}\oplus \FF P^{-1}=T(\FF P^{+1})\oplus
S(\FF P^{+1})\oplus T(\FF P^{-1})\oplus S(\FF P^{-1})\\
&\cong^{(\ref{direct sum})}&T(\FF P^{+1})\oplus
U^{'+1}_{3^{n-2}}\oplus U^{'+1}_{-3^{n-1}}\oplus T(\FF P^{-1})\oplus
U^{'-1}_{3^{n-2}}\oplus
U^{'-1}_{-3^{n-1}}\\
&\cong& 2\FF\oplus X\oplus Y\oplus 2Z,\end{array}$$ where
$X:=U^{'+1}_{3^{n-2}}$, $Y:=U^{'-1}_{3^{n-2}}$, and
$Z:=U^{'+1}_{-3^{n-1}}\cong U^{'-1}_{-3^{n-1}}$
(Lemma~\ref{dimensionO+}). The modules $X, Y$, and $Z$ are simple by
Proposition~\ref{mainO+}.

Next we consider $\ell\mid (3^n-1)$. It is easy to see that $T(\FF
P^\kappa)$ is contained in $U^{'\kappa}_{-3^{n-1}}$ but not in
$U^{'\kappa}_{3^{n-2}}$. We then have $\FF
P^\kappa=U^{'\kappa}_{3^{n-2}}\oplus U^\kappa_{-3^{n-1}},$ where
$U^{'\kappa}_{3^{n-2}}$ is simple and $U^\kappa_{-3^{n-1}}$ is
uniserial (by Proposition~\ref{mainO+}) with composition series
$$0\subset T(\FF P^\kappa) \subset U^{'\kappa}_{-3^{n-1}}\subset
U^\kappa_{-3^{n-1}}.$$ Putting $X:=U^{'+1}_{3^{n-2}}$,
$Y:=U^{'-1}_{3^{n-2}}$, and $Z:=U^{'\kappa}_{-3^{n-1}}/T(\FF
P^\kappa)$, we get
$$\FF P\cong X\oplus Y\oplus 2(\FF-Z-\FF),$$
as described in Table~\ref{tableO+}.
\end{proof}

For the rest of this section, we consider the case
$\ell=\Char(\FF)=2$, where the two graph submodules are the same. We
write $U^\kappa:=U_1^\kappa=U^\kappa_{3^{n-2}}=U^\kappa_{-3^{n-1}}$
and
$U^{'\kappa}:=U^{'\kappa}_1=U^{'\kappa}_{3^{n-2}}=U^{'\kappa}_{-3^{n-1}}$.

\begin{proposition}\label{proofO+2}
Theorem~\ref{theorem} holds when $G=O^+_{2n}(3), n\geq3$ and
$\ell=2$.
\end{proposition}

\begin{proof} Recall that, for $i=0,\pm1$, $\rho^i$ is the permutation character of $G$ afforded by $\CC P^i$. Since $G$ acts with rank 3
on each $P^i$, we have $\rho^i=1+\varphi^i+\psi^i,$ where
$\varphi^i,\psi^i\in\Irr(G)$ and the $\psi^is$ have the same degree
$(3^{2n}-9)/3$ by Lemma~\ref{main2} and the proof of
Lemma~\ref{dimensionO+}. We set $\psi:=\psi^0=\psi^{+1}=\psi^{-1}$.
Note that $\varphi^{+1}(1)=\varphi^{-1}(1)=(3^n-1)(3^{n-1}-1)/8$
from Lemma~\ref{dimensionO+}. Since the smallest degree of a
nonlinear irreducible $2$-Brauer characters of $G$ is
$(3^n-1)(3^{n-1}-1)/8$ (see Theorem~1 of~\cite{Ho}),
$\overline{\varphi^{+1}}$ and $\overline{\varphi^{-1}}$ must be
irreducible.

First we give the proof for $n$ odd. From the study of
$\overline{\rho^0}$ in Corollary~6.5 of~\cite{ST}, we have
$\overline{\psi}=\beta(W)+\beta(X)+\beta(Y)$, where $W, X$, and $Y$
are simple $G$-modules of dimensions $(3^n-1)(3^{n-1}+3)/8-1$,
$(3^n-1)(3^{n-1}-1)/8$, and $(3^n-1)(3^{n-1}-1)/8$, respectively.
Furthermore, $X$ and $Y$ are not isomorphic. Now using
Proposition~\ref{lemmaO+1}(ii) together with the conclusion of the
previous paragraph, we deduce that $U^{'\kappa}$ is isomorphic to
either $X$ or $Y$ and
$\overline{\varphi^\kappa}=\beta(U^{'\kappa})$.

Let $U^{'+1}\cong X$. We wish to show that $U^{'-1}\cong Y$.
Assuming the contrary, we then have $U^{'+1}\cong U^{'-1}\cong X$
and therefore $\overline{\varphi^{+1}}=\overline{\varphi^{-1}}$. Now
we temporarily add subscript $n$ to the standard notations. Then
$\overline{\varphi_n^{+1}}=\overline{\varphi_n^{-1}}$ and hence
$\overline{\rho_n^{+1}}=\overline{\rho_n^{-1}}$. Since $\FF
P_n^\kappa\cong5\FF P_{n-1}^\kappa\oplus 2\FF
P^{-\kappa}_{n-1}\oplus 2\FF P_{n-1}^0\oplus 2\FF$ as $\FF
G_{n-1}$-modules, it follows that
$\overline{\rho_{n-1}^{+1}}=\overline{\rho_{n-1}^{-1}}$. By downward
induction, we get
$\overline{\varphi_3^{+1}}=\overline{\varphi_3^{-1}}$, which is a
contradiction by checking the complex and $2$-Brauer character
tables of $O_6^+(3)$ (see~\cite{Atl1,Atl2}).

We have shown that $U^{'+1}\cong X$ and $U^{'-1}\cong Y$. Notice
that, for $\kappa=\pm1$, $|P^\kappa|\neq 0$ (in $\FF$) and hence
$\FF P^\kappa=T(\FF P^\kappa)\oplus S(\FF P^\kappa)$ and the
composition factors of $S(\FF P^\kappa)$ are $X$ (twice), $Y$, and
$W$. By Proposition~\ref{mainO+} and the self-duality of $S(\FF
P^\kappa)$, the socle series of $S(\FF P^{+1})$ and $S(\FF P^{-1})$
are $X-(Y\oplus W)-X$ and $Y-(X\oplus W)-Y$, respectively, as
described in Table~\ref{tableO+}.

Now we consider $n$ even. In this case,
$\overline{\psi}=1+\beta(W)+\beta(X)+\beta(Y)$, where $W, X$, and
$Y$ are simple $G$-modules of dimensions $(3^n-1)(3^{n-1}+3)/8-2$,
$(3^n-1)(3^{n-1}-1)/8$, and $(3^n-1)(3^{n-1}-1)/8$, respectively.
Repeating the above arguments, we see that $U^{'+1}\cong X$ and
$U^{'-1}\cong Y$.

Notice that $T(\FF P^{+1})\subset S(\FF P^{+1})$, $T(\FF
P^{+1})^\perp=S(\FF P^{+1})$, and  $S(\FF P^{+1})/T(\FF P^{+1})$ is
self-dual and has composition factors: $X$ (twice), $Y$, and $W$.
Again, Proposition~\ref{mainO+} gives the socle series of $S(\FF
P^{+1})/T(\FF P^{+1})$: $X-(W\oplus Y)-X$. The submodule structure
of $\FF P^{+1}$ will be completely determined if we know that of
$U^{'+1\perp}/U^{'+1}$. Note that $U^{'+1\perp}/U^{'+1}$ has
composition factors: $\FF$ (twice), $W$, and $Y$. Using
Lemma~\ref{main1} and inspecting the structure of $\FF P^0$, we see
that $Y$ must be a submodule of $U^{'+1\perp}/U^{'+1}$ but $W$ is
not. Therefore, the structure of $U^{'+1\perp}/U^{'+1}$ is $Y\oplus
(\FF-W-\FF)$.

Similarly, the structure of $S(\FF P^{-1})/T(\FF P^{-1})$ is
$Y-(W\oplus X)-Y$ and that of $U^{'-1\perp}/U^{'-1}$ is $X\oplus
(\FF-W-\FF)$. Now $\FF P$ is determined completely as described in
Table~\ref{tableO+}.
\end{proof}

Propositions~\ref{proofO+1} and \ref{proofO+2} complete the proof of
Theorem~\ref{theorem} for the type ``$+$" orthogonal groups in even
dimension.


\section{The orthogonal groups $O^-_{2n}(3)$}\label{sectionO-}

In this section, we always assume $G=O^-_{2n}(3)$. For
$\kappa=\pm1$, we have
$P^\kappa=\{\langle\sum_{i=1}^n(a_ie_i+b_if_i)\rangle\mid
\sum_{i=1}^{n-1} a_ib_i-a_n^2-b_n^2=\kappa, a_i, b_i\in \FF_3\}$ and
$|P^\kappa|=3^{n-1}(3^n+1)/2$. The parameters of the action of $G$
on $P^\kappa$ are:
$$a=\frac{3^{n-1}(3^{n-1}+1)}{2}, b=3^{2n-2}-1, r=\frac{3^{n-2}(3^{n-1}-1)}{2}, s=\frac{3^{n-1}(3^{n-2}+1)}{2}.$$
The equation~(\ref{quadratic}) now has two roots $-3^{n-2}$ and
$3^{n-1}$. Therefore, $\FF P^\kappa$ has graph submodules
$U^{'\kappa}_{-3^{n-2}}$ and $U^{'\kappa}_{3^{n-1}}$.

\begin{lemma}\label{dimensionO-} If $\ell=\Char(\FF)\neq 2,3$,
then
$$\dim U^{'\kappa}_{-3^{n-2}}=\frac{(3^n+1)(3^{n-1}+1)}{8}, \dim U^{'\kappa}_{3^{n-1}}
=\frac{3^{2n}-9}{8}\text{ and } U^{'1}_{3^{n-1}}\cong U^{'-1}_{3^{n-1}}.$$
\end{lemma}

\begin{proof}
As in the proof of Lemma \ref{dimensionO+}.
\end{proof}

\begin{proposition}\label{proofO-1}
Theorem \ref{theorem} holds when $G=O^-_{2n}(3), n\geq3$ and
$\ell\neq 2,3$.
\end{proposition}

\begin{proof} First we consider $\ell\nmid (3^n+1)$. As in
$\S3$, $\FF P\cong2\FF\oplus X\oplus Y\oplus 2Z,$ where
$X:=U^{'+1}_{-3^{n-2}}$, $Y:=U^{'-1}_{-3^{n-2}}$,
$Z:=U^{'+1}_{3^{n-1}}\cong U^{'-1}_{3^{n-1}}$ and $X,Y,Z$ are simple
by Proposition~\ref{mainO+}.

Second we consider $\ell\mid (3^n+1)$. We have $\FF
P^\kappa=U^{'\kappa}_{-3^{n-2}}\oplus U^\kappa_{3^{n-1}}$, where
$U^{'\kappa}_{-3^{n-2}}$ is simple and $U^\kappa_{3^{n-1}}$ is
uniserial with composition series $0\subset T(\FF P^\kappa) \subset
U^{'\kappa}_{3^{n-1}}\subset U^\kappa_{3^{n-1}}$. Putting
$X:=U^{'+1}_{-3^{n-2}}$, $Y:=U^{'-1}_{-3^{n-2}}$, and
$Z:=U^{'\kappa}_{3^{n-1}}/T(\FF P^\kappa)$, we obtain $\FF P\cong
X\oplus Y\oplus 2(\FF-Z-\FF)$, as stated.
\end{proof}

For the rest of this section, we consider the case $\ell=2$. As
in~$\S$\ref{sectionO+}, we have $\rho^i=1+\varphi^i+\psi$ for
$i=0,\pm1$, where $\varphi^i,\psi\in\Irr(G)$,
$\psi(1)=(3^{2n}-9)/3$, and
$\varphi^{+1}(1)=\varphi^{-1}(1)=(3^n+1)(3^{n-1}+1)/8$. From
Corollary~8.10 of~\cite{ST},
$\overline{\psi}=1+\beta(W)+\beta(X)+\beta(Y)$ when $n$ is even and
$\overline{\psi}=1+1+\beta(W)+\beta(X)+\beta(Y)$ when $n$ is odd.
Here, $X$, $Y$, and $W$ are simple $\FF G$-modules of dimensions
$(3^n+1)(3^{n-1}+1)/8-1$, $(3^n+1)(3^{n-1}+1)/8-1$, and
$(3^n+1)(3^{n-1}-3)/8-1+\delta_{2,n}$, respectively. Moreover,
$X\ncong Y$ and $W$ has smallest dimension among simple $\FF
G$-modules of dimensions greater than 1.

\begin{lemma}\label{lemmaO-3} With the above notation,
\begin{enumerate}
\item[(i)] For $\kappa=\pm1$, $U^{'\kappa}$ is
isomorphic to either $X$ or $Y$ and
$\overline{\varphi^\kappa}=1+\beta(U^{'\kappa})$.
\item[(ii)] If we let $U^{'+1}\cong X$, then $U^{'-1}\cong
Y$.
\end{enumerate}
\end{lemma}

\begin{proof}
(i) We only give here the proof for $n$ even and $\kappa=1$. Other
cases are similar.

Assume the contrary: $U^{'+1}$ is not isomorphic to both $X$ and
$Y$. Lemma~\ref{lemmaO+1} then implies that $U^{'+1}\cong W$ and
$\beta(W)$ is a constituent of $\overline{\varphi^{+1}}$. Hence all
other constituents of $\overline{\varphi^{+1}}$ have degrees at most
$\varphi^{+1}(1)-\dim W=(3^n+1)/2$, which imply that they are linear
since $(3^n+1)(3^{n-1}-3)/8-1+\delta_{2,n}$ is the smallest
dimension of nonlinear irreducible $2$-Brauer characters of $G$.

Let $P^0_1$, $P^{+1}_1$, and $P^{-1}_1$ be the sets of singular
points, plus points, and minus points, respectively, in
$V_1:=\langle e_1,f_1,..., e_{n-1}, f_{n-1}\rangle$. Note that $V_1$
equipped with $Q$ is an orthogonal space of type $+$. Let
$G_1:=O^+_{2n-2}(3)\leq G$. Then we obtain an $\FF G_1$-isomorphism:
\begin{equation}\label{isomorphism} \FF
P^{+1}\cong 2\FF\oplus \FF P_{1}^{+1}\oplus 4\FF P_{1}^{-1}\oplus
4\FF P_1^0.
\end{equation}
Inspecting the structures of $\FF P_1^0$ in Figure~5 of~\cite{ST}
and of $\FF P_1^{+1}$ as well as $\FF P_1^{-1}$ in
Table~\ref{tableO+}, we see that $\FF P^{+1}$, when considered as
$\FF G_1$-module, has 15 composition factors (counting
multiplicities) of dimension 1 (actually all of them are isomorphic
to $\FF$). This contradicts the conclusion of the previous
paragraph.

We have shown that $U^{'+1}$ is isomorphic to either $X$ or $Y$.
Notice from Lemma~\ref{lemmaO+1} that $U^{'+1}$ appears at least
twice as a composition factor of $\FF P^{+1}$. The second statement
of~(i) now follows by comparing the degrees and using
(\ref{isomorphism}).

(ii) Assuming the contrary that $U^{'-1}\ncong Y$, then
$U^{'+1}\cong U^{'-1}\cong X$. It follows that
$\overline{\varphi^{+1}}=\overline{\varphi^{-1}}$ and hence
$\overline{\rho^{+1}}=\overline{\rho^{-1}}$. Therefore, the
isomorphism~(\ref{isomorphism}) together with $\FF P^{-1}\cong
2\FF\oplus \FF P_1^{-1}\oplus 4\FF P_1^{+1}\oplus 4\FF P_1^0$ imply
that the modulo 2 permutation characters afforded by $\FF P_1^{+1}$
and $\FF P_1^{-1}$ are the same. This is a contradiction as seen in
the proof of Proposition~\ref{proofO+2}.
\end{proof}

Since we will use an induction argument to determine the structure
of $\FF P$, we temporarily add the subscript $n$ to our standard
notations. Notice that $G_2=O_4^-(3)$ acts with rank 3 on both
$P_2^{+1}$ and $P_2^{-1}$. Everything we have proved for $n\geq 3$
works exactly the same in the case $n=2$ except that $W_2=0$.

\begin{lemma}\label{lemmaO-1} The structures of $\FF P_2^{+1}$ and $\FF P_2^{-1}$ are given as follows:
\footnotesize
$$\FF P_2^{+1}:\qquad
\xymatrix@C=8pt@R=8pt{
&&X_2\ar@{-}[d]\\
&&\FF\ar@{-}[d]\\
\FF&\oplus&Y_2\ar@{-}[d]\\
&&\FF\ar@{-}[d]\\
&&X_2}
\qquad\qquad\qquad
\FF P_2^{-1}:\qquad
\xymatrix@C=8pt@R=8pt{
&&Y_2\ar@{-}[d]\\
&&\FF\ar@{-}[d]\\
\FF&\oplus&X_2\ar@{-}[d]\\
&&\FF\ar@{-}[d]\\
&&Y_2}$$
\end{lemma}
\normalsize

\begin{proof} By Lemma \ref{lemmaO+1} the module $U'^\kappa$ is simple and self-dual.
The module $\FF P_2^\kappa$ has dimension $15= 3^{2-1}(3^2+1)/2$,
and by Lemma \ref{lemmaO-3} its submodule $U'^\kappa$ has dimension
$4=(3^2+1)(3^{2-1}+1)/8-1$. Therefore by Lemmas \ref{lemmaO+1} and
\ref{main1} the module $U'^\kappa$ appears twice as a composition
factor of $\FF P_2^\kappa$ and once as a composition factor of $\FF
P_2^{-\kappa}$.

As $15$ is odd, $\FF P_2^\kappa = T(\FF P_2^\kappa) \oplus S(\FF
P_2^\kappa)$. Here $U'^\kappa$ is the unique minimal submodule of
the dimension $14$ module $S(\FF P_2^\kappa)$, and
$M=(U'^\kappa)^\perp \cap S(\FF P_2^\kappa)$ its unique maximal
submodule. The submodule $M$ has dimension $10$ and quotient $S(\FF
P_2^\kappa)/M$ isomorphic to $U'^\kappa$. The quotient
$Q=M/U'^\kappa$ is thus self-dual of dimension $6$, possessing two
trivial composition factors in addition to the factor
$U'^{-\kappa}$. There are only three possibilities:
\[
Q = \FF - U'^{-\kappa} - \FF \quad\text{or}\quad Q = \FF \oplus
U'^{-\kappa} \oplus \FF\quad\text{or}\quad Q=(\FF-\FF)\oplus
U'^{-\kappa}.
\]
The first gives the lemma, so we must eliminate the second and
third.

The usual dot product on the natural $\FF_3$-permutation module for
$Sym(6)$ is an invariant bilinear form with radical spanned by the
vector of $1$'s. The action of $Sym(6)$ on the unique nontrivial
composition factor thus gives an injection of $Sym(6)$ into
$O_4^-(3)$. More specifically, $O_4^-(3)\cong 2\times Sym(6)$. With
this in mind, we can choose notation so that the module $\FF
P_2^\kappa$ is the usual permutation module $M^{(4,2)}$ for $Sym(6)$
acting on the $15$ unordered pairs from a set of size six.

As the representation theory of symmetric groups is highly developed
(see, for instance, the elegant treatment in James's book
\cite{james}), the lemma is presumably well known. Indeed the needed
calculations can be done easily, following Example 5.2 of
\cite{james}. We give a short proof, using only some of the
elementary theory.

By an easy calculation and \cite[Cor.~8.5]{james} the Specht
submodule $S=S^{(4,2)}$ of $\FF P_2^\kappa=M^{(4,2)}$ has dimension
$9$, so it must have codimension $1$ in $M$ with $S/U'^\kappa$ of
dimension $5$ in $Q$. Assume (for a contradiction) that $Q = \FF
\oplus U'^{-\kappa} \oplus \FF$ or $Q=(\FF-\FF)\oplus U'^{-\kappa}$.
Then $S/U'^\kappa$ must have shape $\FF \oplus U'^{-\kappa}$. In
particular, the Specht module $S$ has two maximal submodules, one of
codimension $1$ and the other of codimension $4$. But by
\cite[Theorem~4.9]{james}, Specht modules have unique maximal
submodules. This contradiction proves the lemma.
\end{proof}

\begin{lemma}\label{lemmaO-2} For any $n\geq2$, $\FF P_n^{+1}$ does not have any submodule of structure $X_n-Y_n$. Similarly,
$\FF P_n^{-1}$ does not have any submodule of structure $Y_n-X_n$.
\end{lemma}

\begin{proof} Case $n=2$ is clear from Lemma~\ref{lemmaO-1}. So we assume that
$n\geq3$. Let $Q$ be the parabolic subgroup of $G_n$ fixing $\langle
e_1\rangle$. Then $Q=O.L$ where $O=O_3(Q)$, the maximal normal
$3$-subgroup of $Q$ and $L\cong G_{n-1}\times \mathbb{Z}_2$, a Levi
subgroup of $G_n$. Set $V_1:=\langle
e_2,...,e_n,f_2,...,f_n\rangle$. Let $P_1$ be the set of plus points
in $V$ of the form $\langle xe_1+ u\rangle$ and $P_2$ the set of
plus points in $V$ of the form $\langle f_1+xe_1+u\rangle$ with
$x\in \FF_3$ and $u\in V_1$. Then $P_n^{+1}$ is the disjoint union
of $P_1$ and $P_2$.

It is clear that $|P_2|=|O|=3^{2n-2}$ and the stabilizer of $\langle
f_1+e_1\rangle$ in $O$ is trivial. Therefore $O$ acts transitively
on $P_2$. This $O$-orbit is fixed under the action of $G_{n-1}$ on
the set of $O$-orbits on $P^{+1}$. For any plus point $\langle
u\rangle$ in $V_1$, the $O$-orbit of $\langle u\rangle$ consists of
three points: $\langle u\rangle$, $\langle u+e_1\rangle$, and
$\langle u-e_1\rangle$. Hence the action of $G_{n-1}$ on the set of
$O$-orbits in $P_1$ is equivalent to that on the set of plus points
in $V_1$. We have proved the following $\CC G_{n-1}$-isomorphism:
\begin{equation}\label{O-1} C_{\CC P^{+1}_n}(O)\cong \CC
P^{+1}_{n-1}\oplus \CC,
\end{equation} where $C_{\CC
P^{+1}_n}(O)$ is the centralizer of $O$ in $\CC P^{+1}_n$. If $\chi$
is the character of $G_n$ afforded by a module $M$, we denote by
$C_\chi(O)$ the character of $G_{n-1}$ afforded by $C_M(O)$. The
isomorphism (\ref{O-1}) then implies
\begin{equation}\label{O-2}C_{\varphi^{+1}_n}(O)+C_{\psi_n}(O)=
\varphi^{+1}_{n-1}+\psi_{n-1}+1_{G_{n-1}}.\end{equation} From
Frobenius reciprocity,
$$(\psi_n|_Q,1_Q)_Q=(\psi_n,1_Q^{G_n})_{G_n}=(\psi_n,\rho^0_n)_{G_n}>0.$$
Hence, $1_Q$ is a constituent of $\psi_n|_Q$. It follows that
$1_{G_{n-1}}$ is a constituent of $C_{\psi_n}(O)$.

Now we will show that $\psi_{n-1}$ is also a constituent of
$C_{\psi_n}(O)$. Assume not. Then $\psi_{n-1}$ would be a
constituent of $C_{\varphi^{+1}_n}(O)$ by (\ref{O-2}). It follows
that $\overline{\psi_{n-1}}$ is contained in
$C_{\overline{\varphi^{+1}_n}}(O)$. As $\overline{\varphi^{+1}_n}$
is always contained in $\overline{\psi_n}$, we find that
$2\overline{\psi_{n-1}}$ is contained in
$C_{\overline{\varphi^{+1}_n}}(O)+C_{\overline{\psi_n}}(O)$. The
formula~(\ref{O-2}) then implies that $\overline{\psi_{n-1}}$ is
contained in $\overline{\varphi^{+1}_{n-1}}+1$, a contradiction.

We have shown that both $1_{G_{n-1}}$ and $\psi_{n-1}$ are
constituents of $C_{\psi_n}(O)$. Therefore
$C_{\varphi^{+1}_n}(O)=\varphi^{+1}_{n-1}$. It follows by
Lemma~\ref{lemmaO-3} that $C_{X_n}(O)\cong X_{n-1}$. Similarly,
$C_{Y_n}(O)\cong Y_{n-1}$.

Now we prove the lemma by induction. Assuming that the lemma is true
for $n-1$ and supposing the contrary that $\FF P^{+1}_n$ has a
submodule of structure $X_n-Y_n$. Proposition~\ref{mainO+} and the
previous paragraph then show that $C_{(X_n-Y_n)}(O)\cong
X_{n-1}-Y_{n-1}$ is a submodule of $C_{\FF P^{+1}_n}(O)\cong \FF
P^{+1}_{n-1}\oplus \FF$. We deduce that $X_{n-1}-Y_{n-1}$ is a
submodule of $\FF P^{+1}_{n-1}$, contradicting the induction
hypothesis.
\end{proof}

\begin{proposition}\label{proofO-2}
Theorem~\ref{theorem} holds when $G=O^-_{2n}(3),n\geq3$ and
$\ell=2$.
\end{proposition}

\begin{proof}
Using Proposition \ref{mainO+} and Lemma \ref{lemmaO+1}, we see that
the structure of $\FF P^\kappa$ will be determined if we know that
of $U^{'\kappa\perp}/U^{'\kappa}$. We study $U^{'+1\perp}/U^{'+1}$
first.

Consider the case $n$ even. Then, for $\kappa=\pm1$, $|P^\kappa|\neq
0$ and therefore $\FF P^\kappa=T(\FF P^\kappa)\oplus S(\FF
P^\kappa)$.

From the constituents of $\overline{\varphi^{+1}}$ and
$\overline{\psi}$, the composition factors of $U^{'+1\perp}/U^{'+1}$
are: $\FF$ (3 times), $Y$, and $W$. Since $\FF P^\kappa=T(\FF
P^\kappa)\oplus S(\FF P^\kappa)$, we know that
$U^{'+1\perp}/U^{'+1}$ has a direct summand $\FF$. Furthermore,
$X\cong U^{'+1}$ is the socle of $S(\FF P^{+1})$ by Proposition
\ref{mainO+}. Lemma \ref{main1} now implies that
$\Im(Q_{0,+1}|_{S(\FF P^0)})\cong S(\FF P^0)/\Ker(Q_{0,+1}|_{S(\FF
P^0)})$  also has socle $X$. Inspecting the structure of $\FF P^0$
given in Figure~8 of~\cite{ST}, we see that the only quotient of
$S(\FF P^0)$ having $X$ as the socle is $X-W$. This means that
$S(\FF P^{+1})$ has submodule of structure $X-W$ and therefore $W$
is a submodule of $U^{'+1\perp}/U^{'+1}$. By self-duality, $W$ must
be a direct summand of $U^{'+1\perp}/U^{'+1}$.

By Lemma \ref{lemmaO-2}, $Y$ is not a submodule of
$U^{'+1\perp}/U^{'+1}$. Combining this with the previous paragraph,
we conclude that the structure of $U^{'+1\perp}/U^{'+1}$ is
$\FF\oplus W\oplus(\FF-Y-\FF)$.

Now we consider the case $n$ odd. By Lemma \ref{main1},
$\Im(Q_{0,+1})$ is nonzero and different from $T(\FF P^{+1})$. Hence
it has the socle either $X$ or $\FF\oplus X$ by Proposition
\ref{mainO+}. Notice that $\Im(Q_{0,+1})\cong \FF
P^0/\Ker(Q_{0,+1})$. Inspecting the structure of $\FF P^0$ again, we
learn that the structure of $\Im(Q_{0,+1})$ must be $X-\FF-W-\FF$.
It follows that $U^{'+1\perp}/U^{'+1}$ has a submodule of structure
$\FF-W-\FF$. Recall that $U^{'+1\perp}/U^{'+1}$ has composition
factors: $\FF$ (4 times), $Y$, and $W$ and $Y$ is not its submodule
by Lemma \ref{lemmaO-2}. By self-duality, the structure of
$U^{'+1\perp}/U^{'+1}$ is $(\FF-Y-\FF)\oplus (\FF-W-\FF)$.

Arguing similarly for $\kappa=-1$, the structure of
$U^{'-1\perp}/U^{'-1}$ is $\FF\oplus W\oplus(\FF-X-\FF)$ when $n$
even and $(\FF-X-\FF)\oplus (\FF-W-\FF)$ when $n$ odd.
\end{proof}

Propositions~\ref{proofO-1} and \ref{proofO-2} complete the proof of
Theorem~\ref{theorem} for the type $``-"$ orthogonal groups in even
dimension.


\section{The orthogonal groups $O_{2n+1}(3)$}\label{sectionOdd}

In this section, we assume $G=O_{2n+1}(3)$. For $\kappa=\pm1$, we
have $P^\kappa=\{\langle cg+\sum_{i=1}^n(a_ie_i+b_if_i)\rangle\mid
a_i,b_i\in\FF_3,\sum_{i=1}^n a_ib_i-c^2=\kappa\}$ and
$|P^\kappa|=3^n(3^n-\kappa)/2$. The parameters of the action of $G$
on $P^\kappa$ are:
$$a=\frac{3^{n-1}(3^{n}+\kappa)}{2}, b=(3^n+\kappa)(3^{n-1}-\kappa), r=s=\frac{3^{n-1}(3^{n-1}+\kappa)}{2}.$$
Equation~(\ref{quadratic}) now has two roots $-3^{n-1}$ and
$3^{n-1}$. Therefore, for $\kappa=\pm1$, $\FF P^\kappa$ has graph
submodules $U^{'\kappa}_{-3^{n-1}}$ and $U^{'\kappa}_{3^{n-1}}$.

\begin{lemma}\label{dimensionOdd} If $\ell=\Char(\FF)\neq 2,3$,
then $$U^{'+1}_{-3^{n-1}}\cong U^{'-1}_{3^{n-1}}, \dim
U^{'+1}_{-3^{n-1}}=\dim U^{'-1}_{3^{n-1}}=\frac{3^{2n}-1}{4},$$
$$\dim U^{'-1}_{-3^{n-1}}=\frac{(3^n-1)(3^{n}+3)}{4}, \text{ and }\dim U^{'+1}_{3^{n-1}}=\frac{(3^{n}+1)(3^n-3)}{4}.$$
\end{lemma}

\begin{proof} This is similar to the proof of Lemma~\ref{dimensionO+}.
We remark in this case that $U^{'-1}_{-3^{n-1}}$ and
$U^{'+1}_{3^{n-1}}$ are isomorphic to graph submodules of $\FF P^0$.
\end{proof}

\begin{proposition}\label{proofOdd1}
Theorem~\ref{theorem} holds when $G=O_{2n+1}(3),n\geq3$ and
$\ell\neq 2,3$.
\end{proposition}

\begin{proof}
\textbf{Case 1}: $\ell\nmid (3^n-1), \ell\nmid (3^n+1)$. In this
case, $\FF P\cong2\FF\oplus X\oplus Y\oplus 2Z,$ where
$X:=U^{'+1}_{3^{n-1}}$, $Y:=U^{'-1}_{-3^{n-1}}$, and
$Z:=U^{'+1}_{-3^{n-1}}\cong U^{'-1}_{3^{n-1}}$. By
Proposition~\ref{mainO+}, $X, Y$, and $Z$ are simple.

\medskip

\textbf{Case 2}: $\ell\mid (3^n-1)$. We have $$\FF P^{-1}=T(\FF
P^{-1})\oplus U^{'-1}_{-3^{n-1}}\oplus U^{'-1}_{3^{n-1}}\cong
\FF\oplus Y\oplus Z,$$ where $Y:=U^{'-1}_{-3^{n-1}}$,
$Z:=U^{'1}_{-3^{n-1}}\cong U^{'-1}_{3^{n-1}}$ and
$$\FF P^{+1}=U^{+1}_{3^{n-1}}\oplus U^{'+1}_{-3^{n-1}}\cong U^{+1}_{3^{n-1}}\oplus Z,$$
where $U^{+1}_{3^{n-1}}$ is uniserial with composition series
$0\subset T(\FF P^{+1}) \subset U^{'+1}_{3^{n-1}}\subset
U^{+1}_{3^{n-1}}.$ Setting $X:=U^{'+1}_{3^{n-1}}/T(\FF P^{+1})$, we
get
$$\FF P\cong \FF\oplus (\FF-X-\FF)\oplus Y\oplus 2Z.$$

\medskip

\textbf{Case 3}: $\ell\mid (3^n+1)$. As in Case~2,
$$\FF P\cong \FF\oplus X\oplus (\FF-Y-\FF)\oplus 2Z,$$ where
$X:=U^{'+1}_{3^{n-1}}$, $Y:=U^{'-1}_{-3^{n-1}}/T(\FF P^{-1})$, and
$Z:=U^{'+1}_{-3^{n-1}}\cong U^{'-1}_{3^{n-1}}$.
\end{proof}

Now we consider the case $\ell=2$. Following
Lemma~\ref{dimensionOdd}, we assume that
$\rho^\kappa=1+\varphi^\kappa+\psi$ for $\kappa=\pm1$, where
$\varphi^\kappa,\psi\in\Irr(G)$,
$\varphi^\kappa(1)=(3^n+\kappa)(3^n-\kappa)/4$, and
$\psi(1)=(3^{2n}-1)/4$. Then $\rho^0=1+\varphi^{+1}+\varphi^{-1}$.
From Corollary~7.5 of~\cite{ST}, we have
$\overline{\varphi^{+1}}=\chi+\beta(X_1)$ and
$\overline{\varphi^{-1}}=1+\chi+\beta(Y_1)$, where $\chi$ is a
$2$-Brauer character of $G$ and $X_1,Y_1$ are simple $G$-modules of
dimensions $(3^n-1)(3^n-3)/8$, $(3^n+1)(3^n+3)/8-1$, respectively.
Furthermore, $\chi=\beta(Z_1)$ if $n$ is odd and $\chi=1+\beta(Z_1)$
if $n$ is even, where $Z_1$ is a simple module of dimension
$(3^{2n}-9)/8-\delta_{2,n}$. The following lemma gives the
decomposition of $\overline{\psi}$ into irreducible $2$-Brauer
characters of $G$.

\begin{lemma}\label{lemmaOdd1} With the above notation, $U^{'1}\cong
X_1$ and $U^{'-1}\cong Y_1$. Consequently,
$\overline{\psi}=\beta(X_1)+\beta(Y_1)$.
\end{lemma}

\begin{proof} By Proposition~\ref{mainO+} and Lemma \ref{main1}, the submodule
$\Im(Q_{0,+1})$ of $\FF P^{+1}$ has $U^{'+1}$ as a composition
factor. It follows that $U^{'+1}\in\{X_1, Y_1, Z_1\}$ since
$\Im(Q_{0,+1})\cong \FF P^0/\Ker(Q_{0,+1})$ and the composition
factors of $\FF P^0$ are: $\FF$ (twice or four times), $X_1$, $Y_1$,
and $Z_1$ (twice) (see Figure~6 of~\cite{ST}).

Set $G_1:=O^+_{2n}(3)\leq G$. Let $P^{+1}_1$ and $P^{-1}_1$ be the
sets of plus points and minus points in $\langle
e_1,...,e_n,f_1,...,f_n\rangle$. Since $P^{+1}=P^{+1}_1\cup
\{\langle v+g\rangle\mid \langle v\rangle\in P^{-1}_1\}\cup
\{\langle v-g\rangle\mid \langle v\rangle\in P^{-1}_1\}$, we have
$\FF P^{+1}|_{G_1}\cong \FF P_1^{+1}\oplus 2\FF P^{-1}_1$. Moreover,
if $P^{-1}_1+g:=\{\langle v+g\rangle\mid \langle v\rangle\in
P^{-1}_1\}$ and $P^{-1}_1-g:=\{\langle v-g\rangle\mid \langle
v\rangle\in P^{-1}_1\}$ then we get an $\FF G_1$-isomorphism:
$$U^{+1}|_{G_1}\cong \langle v_{1,\alpha}\mid \alpha\in P^{+1}_1\rangle\oplus \langle
v_{1,\alpha}\mid \alpha\in P^{-1}_1+g\rangle \oplus\langle
v_{1,\alpha}\mid \alpha\in P^{-1}_1-g\rangle,$$ where all summands
are clearly nonzero and nontrivial $\FF G_1$-modules. These summands
are submodules of $\FF P_1^{+1}\oplus 2\FF P^{-1}_1$. It follows
that, by Proposition~\ref{mainO+}, each of them contains a graph
submodule of $\FF P_1^{+1}$ or $\FF P_1^{-1}$. Notice that the
dimensions of the graph submodules of $\FF P_1^1$ as well as $\FF
P_1^{-1}$ are $(3^n-1)(3^{n-1}-1)/8$.

We have shown that $U^{+1}|_{G_1}$ has 3 composition factors
(counting multiplicities) of degree $(3^n-1)(3^{n-1}-1)/8$. If
$U^{+1}|_{G_1}$ has another nonlinear composition factor, $\dim
U^{+1}$ would be at least $4(3^n-1)(3^{n-1}-1)/8+1$ since the
smallest degree of nonlinear irreducible $2$-Brauer character of
$G_1$ is $(3^n-1)(3^{n-1}-1)/8$ (see Table~1 of~\cite{Ho}). This
contradicts the fact that $U^{'+1}\in\{X_1, Y_1, Z_1\}$, whence
$U^{+1}|_{G_1}$ has exactly 3 nonlinear composition factors, all of
degree $(3^n-1)(3^{n-1}-1)/8$.

Recall that $U^{+1}|_{G_1}$ is a submodule of $\FF P_1^{+1}\oplus
2\FF P^{-1}_1$. It follows that $U^{+1}_{G_1}$ has at most 6
composition factors of dimension 1 (see Table~\ref{tableO+}).
Combining this with the conclusion of the previous paragraph, we
obtain $\dim U^{+1}\leq 3(3^n-1)(3^{n-1}-1)/8+6$. This forces $\dim
U^{+1}=\dim X_1+1$ and therefore $U^{'+1}\cong X_1$ again by
$U^{'+1}\in\{X_1, Y_1, Z_1\}$.

The arguments for $U^{'-1}\cong Y_1$ are similar. Since
$U^{'\kappa}$ appears at least twice as a composition factor of $\FF
P^\kappa$ (see Lemma~\ref{lemmaO+1}), both $U^{'+1}$ and $U^{'-1}$
are constituents of $\overline{\psi}$. Therefore
$\overline{\psi}=\beta(X_1)+\beta(Y_1)$ by comparing degrees.
\end{proof}

\begin{proposition}\label{proofOdd2}
Theorem~\ref{theorem} holds when $G=O_{2n+1}(3), n\geq3$ and
$\ell=2$.
\end{proposition}

\begin{proof}
\textbf{Case 4}: $n$ even. We know from Lemma~\ref{lemmaOdd1} that
$U^{'+1}\cong X_1$. The self duality of~$U^{'+1}$ from
Lemma~\ref{lemmaO+1} then implies that $U^{'+1\perp}/U^{'+1}$ has
composition factors: $\FF$ (twice), $Z_1$, and $Y_1$. Using
Lemma~\ref{main1} and inspecting the structure of $\FF P^0$ (see
Figure~6 of~\cite{ST}), we see that $Y_1$ must be a submodule of
$U^{'+1\perp}/U^{'+1}$ but $Z_1$ is not. Therefore, the structure of
$U^{'+1\perp}/U^{'+1}$ is $Y_1\oplus (\FF-Z_1-\FF)$ and hence that
of $\FF P^1$ is determined.

Now we determine the structure of $\FF P^{-1}$. Since
$[P^{-1}]\notin S(\FF P^{-1})$, $\FF P^{-1}=T(\FF P^{-1})\oplus
S(\FF P^{-1})$. Also, $U^{'-1}\cong Y_1$ is the socle of $S(\FF
P^{-1})$. Since $S(\FF P^{-1})$ is self-dual, its head is also
(isomorphic to) $Y_1$. Hence $S(\FF P^{-1})/\Ker(Q_{-1,0}|_{S(\FF
P^{-1})})\cong\Im(Q_{-1,0}|_{S(\FF P^{-1})})$ has $Y_1$ as head.
From the submodule structure of $\FF P^0$, we find that
$\Im(Q_{-1,0}|_{S(\FF P^{-1})})$ is uniserial with socle series
$\FF-Z_1-\FF-Y_1$. So $S(\FF P^{-1})$ has a quotient
$\FF-Z_1-\FF-Y_1$. Again by its self-duality, it has a submodule
$Y_1-\FF -Z_1-\FF$, which implies that $U^{'-1\perp}/U^{'-1}$ has a
submodule $\FF -Z_1-\FF$. Notice that $U^{'-1\perp}/U^{'-1}$ is
self-dual and has composition factors: $\FF$ (twice), $X_1$, and
$Z_1$. Its structure must be $\FF \oplus X_1\oplus (\FF -Z_1-\FF)$,
as described in Table~\ref{tableOdd}.

\textbf{Case 5}: $n$ odd. First we find the structure of $\FF
P^{+1}$. Composition factors of $U^{'+1\perp}/U^{'+1}$ are $\FF$,
$Y_1$, and $Z_1$. Therefore, the structure of $U^{'+1\perp}/U^{'+1}$
is simply $\FF\oplus Y_1\oplus Z_1$ by its self-duality.

Now we turn to $\FF P^{-1}$. By Proposition~\ref{mainO+}, the socle
of $\Im(Q_{0,-1})$ is either $Y_1$ ($\cong U^{'-1}$) or $\FF\oplus
Y_1$ ($\cong U^{-1}$). Notice that $\Im(Q_{0,-1})\cong \FF P^0/\Ker
Q_{0,-1}$ and $\FF P^0$ has only one quotient having such socle,
which is $Y_1-(\FF\oplus Z_1)$ (see the structure of $\FF P^0$ in
Figure~6 of~\cite{ST}). We deduce that $\FF P^{-1}$ has a submodule
of structure $Y_1-(\FF\oplus Z_1)$. We temporarily set $\FF_1:=T(\FF
P^{-1})$ and $\FF_2:=\FF P^{-1}/S(\FF P^{-1})$. Then the submodule
of $\FF P^{-1}$ of structure $Y_1-(\FF\oplus Z_1)$ must be
$Y_1-(\FF_2\oplus Z_1)$. It follows that $U^{'-1\perp}/U^{'-1}$ has
a submodule $\FF_2\oplus Z_1$. Recall that $U^{'-1\perp}/U^{'-1}$ is
self-dual and has composition factors: $\FF_1$, $\FF_2$, $Z_1$, and
$X_1$, we conclude that its structure is $\FF_1\oplus\FF_2\oplus
Z_1\oplus X_1$.
\end{proof}

Propositions~\ref{proofOdd1} and \ref{proofOdd2} complete the proof
of Theorem~\ref{theorem} for the orthogonal groups in odd dimension.

\medskip

\textbf{Acknowledgement}: The authors are grateful to Ulrich
Meierfrankenfeld for his helpful suggestions leading to the proof of
Lemma \ref{lemmaO-2}. Part of this work was done while the second
author was a postdoctoral fellow in the Department of Mathematics at
Michigan State University. It is a pleasure to thank the department
for the supportive and hospitable work environment.


\end{document}